\documentclass[12pt]{amsart}

\usepackage{amssymb,amsfonts, amsmath, amsaddr, dutchcal, mathtools}

\usepackage{hyperref}
\usepackage{xcolor} 
\usepackage[normalem]{ulem} 
\usepackage{enumerate}
\usepackage{amsthm}
\usepackage{thmtools}
\usepackage{cleveref}

\newtheorem{thm}{Theorem}[section]
\newtheorem{lemma}[thm]{Lemma}
\newtheorem{corollary}[thm]{Corollary}
\newtheorem{proposition}[thm]{Proposition}

\theoremstyle{definition}

\newtheorem{defn}[thm]{Definition}
\newtheorem{example}[thm]{Example}
\newtheorem{remark}[thm]{Remark}

\newcommand{\enclosepart}[1]{(#1)}
\newcommand{\partref}[1]{\enclosepart{\ref{#1}}}

\def\N{\mathbb{N}}
\def\R{\mathbb{R}}
\def\C{\mathbb{C}}
\def\L{\mathbb{L}}

\def\lam{\lambda}

\def\eps{\varepsilon}
\def\phi{\varphi}
\def\Eps{\mathcal{E}}
\def\0{{\bf 0 }}
\def\1{{\bf 1}}

\newcommand{\upC}{{\mathrm{C}}}

\newcommand{\dis}[1]{\displaystyle{#1}}
\newcommand{\norm}[1]{\left\Vert #1 \right\Vert}
\newcommand{\abs}[1]{\left\vert #1 \right\vert}

\newcommand{\Dc}{De\-de\-kind com\-plete}
\newcommand{\sDc}{$\sigma$-De\-de\-kind com\-plete}

\newcommand{\set}[1]{\{#1\}}

\newcommand{\pos}[1]{{#1^+}}
\newcommand{\negt}[1]{{#1^-}}
\newcommand{\zerofunction}{\textbf{0} }
\newcommand{\onefunction}{\textbf{1}}
\newcommand{\indicator}[1]{\chi_{#1}}
\newcommand{\unit}{1}

\newcommand{\seq}[1]{ ( {#1}_n )}

\newcommand{\pnorm}[1]{\left\Vert #1 \right\Vert_p}
\newcommand{\qnorm}[1]{\left\Vert #1 \right\Vert_q}
\newcommand{\enorm}[1]{\left\Vert #1 \right\Vert_{\infty}}

\newcommand{\pset}{S}
\newcommand{\stone}{K}

\def\lam{\lambda}
\newcommand{\ls}{\L}
\newcommand{\posls}{\pos{\ls}}
\newcommand{\lsext}{\overline{\ls}}
\newcommand{\poslsext}{\overline{\posls}}

\newcommand{\salg}{\Sigma}
\newcommand{\mss}{A}

\newcommand{\ms}{(\pset,\salg)}
\newcommand{\msm}{(\pset,\salg,\npm,\L)}

\newcommand{\npm}{\mu}

\newcommand{\poslstepfun}{{{\mathcal S}(\pset,\salg;\pos{\ls_1})}}
\newcommand{\lstepfun}{{{\mathcal S}(\pset,\salg;\ls_1)}}
\newcommand{\poslmeasfun}{{{\mathcal M}(\pset,\salg;\pos{\ls})}}
\newcommand{\poslmeasfunext}{{{\mathcal M}(\pset,\salg;\poslsext)}}
\newcommand{\lmeasfun}{{{\mathcal M} (\pset,\salg;\ls)}}
\newcommand{\integrablelfun}{{{\mathcal L}^1(\pset,\salg,\npm;\ls)}}
\newcommand{\posintegrablelfun}{{{\mathcal L}^1(\pset,\salg,\npm;\posls)}}
\newcommand{\pintegrablelfun}{{{\mathcal L}^p(\pset,\salg,\npm;\ls)}}
\newcommand{\qintegrablelfun}{{{\mathcal L}^q(\pset,\salg,\npm;\ls)}}
\newcommand{\eblfun}{{{\mathcal L}^{\infty}(\pset,\salg,\npm;\ls)}}
\newcommand{\kerone}{{{\mathcal N}_1(\pset,\salg,\npm;\ls)}}
\newcommand{\kerp}{{{\mathcal N}_p(\pset,\salg,\npm;\ls)}}
\newcommand{\kerq}{{{\mathcal N}_q(\pset,\salg,\npm;\ls)}}
\newcommand{\kerall}{\mathcal{N}}
\newcommand{\aezero}{{{\mathcal N}_{\infty}(\pset,\salg,\npm;\ls)}}

\newcommand{\ellp}{{{\mathrm L}^p(\pset,\salg,\npm;\ls)}}

\newcommand{\ebllq}{{{\mathrm L}^{\infty}(\pset,\salg,\npm;\ls)}}

\newcommand{\cont}[1]{{\upC}(#1)}
\newcommand{\continfty}[1]{\upC_{\infty}(#1)}

\newcommand{\contstone}{\cont{\stone}}
\newcommand{\continftystone}{\continfty{\stone}}

\def\d{\,\mathrm{d}}

\newcommand{\comment}[1]{}

\numberwithin{equation}{section}

\begin{document}
	
\title{$\L$-valued integration}

\author{Xingni Jiang}
\email{x.jiang@scu.edu.cn}
\address{College of Mathematics, Sichuan University, No. 24, South Section, First Ring Road, Chengdu, People’s Republic of China}

\author{Jan Harm van der Walt}
\email{janharm.vanderwalt@up.ac.za}
\address{Department of Mathematics and Applied Mathematics, University of Pretoria, Private Bag X20 Hatfield, Pretoria 0028, South Africa}

\author{Marten Wortel}
\email{marten.wortel@up.ac.za}
\address{Department of Mathematics and Applied Mathematics, University of Pretoria, Private Bag X20 Hatfield, Pretoria 0028, South Africa}

\thanks{This work is based on the research supported in part by the National Research Foundation of South Africa (Ref Number CPRR240418214705).}

	
	\maketitle

\begin{abstract}
We develop integration theory for functions taking values in a Dedekind complete unital $f$-algebra $\L$ with respect to $\L$-valued measures. We then discuss and prove completeness results of $\L$-valued $L^p$-spaces.
\end{abstract}

\section{Introduction}

Let $\L$ be a Dedekind complete unital $f$-algebra. In \cite{L-functional_analysis}, $\L$-normed spaces were investigated. These are $\L$-modules equipped with an $\L^+$-valued norm satisfying the usual properties. Convergence in an $\L$-normed space $X$ is determined by order convergence in $\L$, in the sense that a net $(x_\alpha)$ in $X$ is defined to converge to $x \in X$ whenever $\norm{x_\alpha - x}$ order converges to $0$ in $\L$. One can then define Cauchy nets and completeness, and $\L$-valued `sequence' spaces $\ell^p(S,\L)$ (generalized to $S$ being any nonempty set rather than $\N$) were investigated in detail: completeness and the usual duality results were proven.

In this paper we develop a corresponding theory of $\L$-valued $L^p$-spaces, for which it is clearly necessary to first developed a theory of integration in this setting. We have to consider $\L$-valued functions, since real-valued functions have no natural $\L$-module structure, and $\L$-valued functions clearly do. Furthermore, in order to generalize for example the duality $c_0(\N, \L)^* \cong \ell^1(\N, \L)$ (see \cite[Corollary~4.0.12]{L-functional_analysis}), it will be necessary to consider $\L$-valued measures as well. Therefore both our functions and measures will be vector-valued, which deviates from many approaches in the literature where either the functions or the measures are real-valued.

One interesting approach in the literature in which both functions and measures are vector-valued is the approach by Grobler and \linebreak Labuschagne in \cite[Section~4]{grobler_labuschagne_ito_integral}, where an integration theory is developed that is specialized to the setting of stochastic processes in vector lattices. In this theory the vector lattice satisfies addition assumptions, such as perfectness and the existence of a conditional expectation (as well as a technical assumption of universal completeness with respect to the conditional expectation). Using the associated order continuous functionals, Grobler and Labuschagne then obtain certain Banach lattices where they employ and refine known techniques from Banach space-valued integration theory, specifically the Dobrakov integral. This approach, however, is not well-suited to our more basic setting of a Dedekind complete unital $f$-algebra.

This paper largely follows the approach of \cite{de_jeu_jiang:2022a}, in which a comprehensive theory of integration of real-valued functions with respect to measures taking values in a suitable partially ordered vector space is developed. Indeed, many of our results closely parallel the techniques developed in \cite{de_jeu_jiang:2022a}. However, there is one important exception: \cite[Lemma~6.2]{de_jeu_jiang:2022a}, which is crucial for defining the integral of measurable functions, has no obvious extension to $\L$-valued functions. Indeed, in certain natural cases the analogue of \cite[Lemma~6.2]{de_jeu_jiang:2022a} fails, see \Cref{e:non_proper_measure}. We solve this by introducing proper measures in an approach similar to that in \cite{wickstead:1982}.

\medskip

This paper is organized as follows. \Cref{sec:preliminaries} discusses properties of the Dedekind complete unital $f$-algebra $\L$ as well as other preliminaries on local convergence in measure and $\L$-valued measures.

In \Cref{sec:measurable_functions_and_measurable_step_functions}, we discuss measurable functions and (bounded) measurable step functions. We show that measurable functions can be approximated by bounded measurable step functions.

We use the above result in \Cref{sec:proper_measures} to show that there is a relatively large class of measures which are proper. Proper measures are defined as those measures for which the analogue of \cite[Lemma~6.2]{de_jeu_jiang:2022a} holds.

The theory is then developed smoothly for proper measures in \Cref{sec:integration}, closely following the corresponding theory in \cite{de_jeu_jiang:2022a}. We conclude this section by obtaining some properties of the space of integrable functions.

The important convergence theorems, like the Dominated Convergence Theorem, the Monotone Convergence Theorem, and Fatou's Lemma, are proven in \Cref{sec:properties_of_the_intergal}.  Again, this section closely follows the theory developed in \cite{de_jeu_jiang:2022a}.

These results are used to define $\L$-valued $L^p$-spaces in \Cref{sec:lp-spaces} and study their properties. The usual norm (in)equalities are verified, and the section ends by showing that such $L^p$-spaces have the property that every absolutely convergent series converges.

In general it is not clear whether for an $\L$-normed space $X$, the property that every absolutely convergent series converges implies that $X$ is complete, as in the classical case. However, in \Cref{sec:sequential_completeness} we show that this does hold if $\L$ satisfies the countable supremum property, and so we conclude that our $L^p$-spaces are complete under this assumption.

\section{Preliminaries}\label{sec:preliminaries}

\subsection{Dedekind complete unital $f$-algebras}

Let $\ls$ be a \Dc\ unital $f$-algebra; here ``unital'' means $\ls$ contains a multiplicative unit which we denote by $\unit$. By \cite[Theorem~10.7]{de_pagter_THESIS:1981}, $\unit$ is automatically a weak order unit of $\ls$. We denote by $\L_1$ the order ideal in $\L$ generated by 1. Combining \cite[Theorem~2.64]{aliprantis_burkinshaw_POSITIVE_OPERATORS_SPRINGER_REPRINT:2006} and \cite[Theorem~7.29]{aliprantis_burkinshaw_LOCALLY_SOLID_RIESZ_SPACES_WITH_APPLICATIONS_TO_ECONOMICS_SECOND_EDITION:2003} shows that there exists a unique Stonean space $\stone$ such that $\ls$ is an order dense sublattice and subalgebra of $\continftystone$, and $\unit$ corresponds to the constant one function $\onefunction_\stone$ in $\contstone$. Moreover $\ls$ is an order ideal in $\continftystone$, see \cite[Theorem~1.40]{aliprantis_burkinshaw_LOCALLY_SOLID_RIESZ_SPACES_WITH_APPLICATIONS_TO_ECONOMICS_SECOND_EDITION:2003}. Hence we have that, as in \cite[Remark~2.1.1]{L-functional_analysis},

\begin{equation}\label{e:L_inside_continuous_functions}
\contstone = \L_1 \subseteq \L \subseteq \continftystone.
\end{equation}
This representation allows us to define $\lam^p$ for $0 \leq p < \infty$ pointwise in $C_\infty(K)$ and therefore also in $\L$ (cf.\ \cite[Section~2.1]{L-functional_analysis}). By \cite[Lemma~4.0.2$(ii)$]{L-functional_analysis}, $\lam \mapsto \lam^p$ is order continuous as a map from $\L$ to $\L$.

We identify every real number $c$ with $c1$ in $\L$ so that $\R$ is viewed as a subalgebra and sublattice of $\L$.  We write $c$ instead of $c1$, so that $c\in\L$.

As in \cite{de_jeu_jiang:2022a}, we extend the partial ordering on $\ls$ by adding a new element $\infty$, and declaring that $\lambda\leq \infty$ for all $\lambda\in \ls$.  We denote $\lsext=\ls\cup\set{\infty}$, $\poslsext=\posls \cup\set{\infty}$, and the elements of $\lsext$ that are in $\ls$
will be called \emph{finite}. Clearly every nonempty subset of $\lsext$ has a supremum.  In particular, $\poslsext$ is a complete lattice, containing $\posls$ as a sublattice.  Moreover, for any $A\subseteq \posls$ and $\lambda\in \posls$, $\lambda \sup A$ in $\posls$ if and only if $\lambda= \sup A$ in $\poslsext$.
\begin{remark}
    There are more sophisticated ways of adding infinite elements to $\L$: one can consider the sup-completion $\L^s$ of $\L$, which was first introduced by Donner in \cite{donner_sup_completion}. This is used in \cite{grobler_daniell_integral} to construct a vector-valued integration theory using the Daniell integral. However, for our purposes it suffices to consider $\lsext$.
\end{remark}
On $\lsext$, we define  $\infty + \lambda = \lambda + \infty = \infty$ for all $\lambda\in \lsext$. For the multiplication, we define $\infty \cdot \lambda = \lambda \cdot \infty = \infty$ for any $\lambda\in\poslsext\setminus\set{0}$ and $0 \cdot \infty = \infty \cdot0=0$. It is easy to verify that for any $\lambda_1$, $\lambda_2$, $\alpha_1$ and $\alpha_2$ in $\poslsext$ with $\lambda_1\leq\lambda_2$ and $\alpha_1 \leq \alpha_2$, we have $0\leq \lambda_1\alpha_1\leq\lambda_2\alpha_2$. Note that $\lsext$ does not satisfy the $f$-algebra property; for example, if $\lambda$ and $\alpha$ in $\posls\setminus\set{0}$ satisfy  $\lambda\wedge\alpha=0$, then $[\lambda\cdot \infty ]\wedge\alpha= \infty \wedge\alpha=\alpha\neq 0$.  We remark that this will not cause any trouble for us, since the measures we consider are finite, and the integrable functions are almost everywhere finite, see \Cref{res:positive_integrable_is_almost_everywhere_finite}. Hence, in the end, we still work in $\ls$ almost everywhere.

\subsection{Summability in $\L$}\label{Sec:summability}

Let $(\lambda_\alpha)_{\alpha \in I}$ be a net in $\L$.  Denote by $\mathcal{F}_I$ the collection of finite subsets of $I$, ordered by inclusion.  The net $(\lambda_\alpha)_{\alpha \in I}$ is called summable if the net of finite partial sums
\[
\left(\sum_{\alpha \in F}\lambda_\alpha\right)_{F\in \mathcal{F}_I}
\]
is (order) convergent in $\ls$.  In this case, we denote the (order) limit of this net by $\displaystyle{\sum}\lambda_\alpha$.  Observe that if $\lambda_\alpha \geq 0$ for every $\alpha \in I$ then
\[
\sum \lambda_\alpha = \sup \left\{ \sum_{\alpha \in F}\lambda_\alpha ~:~ F\in\mathcal{F}_I\right\}.
\]

\subsection{Sequences of functions}

Let $S$ be a set, let $\seq{f}$ be a sequence of functions from $S$ to a Dedekind complete lattice $E$, such as $\L$, $\L^+$ or $\lsext$, and let $f \colon S \to E$. If for every $s \in S$, $(f_n(s))_{n=1}^\infty$ is an increasing sequence in $E$ with $\sup_{n \in \N} f_n(s) = f(s)$ in $E$, we write $f_n \uparrow f$ pointwise. The notation $f_n \downarrow f$ pointwise is defined similarly.  A sequence $(\lambda_n)$ in $E$ order converges to $\lambda\in E$ if $\displaystyle\limsup_{n} \lambda_n = \lambda = \liminf_n \lambda_n$, see for instance \cite[Definition 2.1]{Preuss}\footnote{Under our assumption of Dedekind completeness, this is consistent with the definition of order convergence in vector lattices, see for instance \cite{OBrien_vanderwalt_troitsky_convergence_structures}.  We also note that the definition in \cite{Preuss} is given for filters, but it is easily seen to reduce to our formulation in the case of sequences.}. We write $f_n\to f$ to mean that for every $s\in S$, the sequence $(f_n(s))$ order converges to $f(s)$ in $E$.

\subsection{Local convergence in measure}

In certain cases, $\L$ can be idenfied with a space of real-valued measurable functions on a measure space.  The next remark, which summarizes \cite[Remark~2.1.4]{L-functional_analysis}, explains this connection.
 \begin{remark}\label{r:hyperstonean}
In \eqref{e:L_inside_continuous_functions}, the order continuous real-valued functionals on $C(K)$ separate the points of $C(K)$ if and only if $K$ is \emph{hyper-Stonean} (i.e. the union of the supports of the normal measures on $K$ is dense in $K$), if and only if $C(K)_\C$ is a von Neumann algebra, if and only if $C(K)$ can be represented as a space $L^\infty(\Omega, \mathcal{F}, m) = L^\infty(m)$ of essentially bounded measurable functions with respect to a decomposable measure $m$ on a locally compact Hausdorff space, see for instance the survey paper \cite{Blecher_Goldstein_Labuschagne_abelian_von_neumann_algebras}. In this case $C_\infty(K)$ can be identified with the space of all measurable functions $L^0(m)$, so that \[
L^\infty(m) = \L_1 \subseteq \L \subseteq L^0(m).
\]
\end{remark}
In the rest of this subsection, let $(\Omega, \mathcal{F}, m)$ be a measure space. We recall the well-known concept of local convergence in measure.

\begin{defn}\label{d:local_convergence_in_measures}
A sequence $(f_n)$ in $L^0(m)$ converges to $f\in L^0(m)$ locally in measure if for every $F\in\mathcal{F}$ such that $\mu(F)<\infty$ and for every $\epsilon>0$,
\[
\lim_n \mu(\left|f-f_n\right|^{-1}([\epsilon,\infty])\cap F)=0.
\]
\end{defn}

The next result is a combination of \cite[245D, 245E \& 245Y (d)]{fremlin_MEASURE_THEORY_VOLUME_2:2003}.  Although Fremlin does not state that the topology of local convergence in measure is locally solid, this is obvious.  In \cite[page 210]{aliprantis_burkinshaw_LOCALLY_SOLID_RIESZ_SPACES_WITH_APPLICATIONS_TO_ECONOMICS_SECOND_EDITION:2003} this fact is stated for the case of a $\sigma$-finite measure.
\begin{thm}\label{t:L^0_has_nice_topology}
Local convergence in measure is induced by an order continuous locally solid topology $\tau$ on $L^0(m)$. \begin{enumerate}
    \item[(i)] $\tau$ is Hausdorff if and only if $m$ is semi-finite.
    \item[(ii)] $\tau$ is metrizable if and only if $m$ is $\sigma$-finite.
    \item[(iii)] $\tau$ is Hausdorff and $(L^0(m), \tau)$ is complete if and only if $m$ is localizable.
\end{enumerate}
\end{thm}
It is well-known (see for example \cite{luxemburg_zaanen_RIESZ_SPACES_VOLUME_I:1971}) that for a $\sigma$-finite measure $m$, $L^0(m)$ has the countable sup property; i.e. every set which has a supremum has a countable subset with the same supremum.  In fact, if $m$ is semi-finite, then $L^0(m)$ has the countable sup property if and only if $m$ is $\sigma$-finite, see \cite[Proposition 6.5]{kandic_taylor:2018}

\subsection{$\L$-valued measures}

 Let $\ms$ be a measurable space, i.e., a set $S$ equipped with a $\sigma$-algebra $\salg$. A finite $\L$-valued measure is a map $\mu \colon \salg \to \L^+$ so that $\mu(\emptyset) = 0$, and, if $(\mss_n)$ is a disjoint sequence in $\salg$, then
 \[
 \mu \left( \bigcup_{n=1}^\infty \mss_n \right) = \sum_{n=1}^\infty \mu(\mss_n).
 \]
 This definition is a special case of \cite[Definition~4.1]{de_jeu_jiang:2022a}. In this paper $\msm$ always denotes a measurable space $\ms$ equipped with a finite $\L$-valued measure $\mu$.

\section{Measurable functions and measurable step functions}\label{sec:measurable_functions_and_measurable_step_functions}

We recall that $\msm$ denotes a measurable space $\ms$ equipped with a finite $\L$-valued measure $\mu$.  For $A\in\Sigma$, we denote by $\chi_A$ the $\L$-valued characteristic function of $A$. A function $\phi \colon \pset \to \L$  is called a \emph{measurable step  function} if its range is a finite set $\{\lambda_1,\ldots,\lambda_n\}$, and for every $i\in\{1,\ldots,n\}$, the set $A_i \coloneqq \psi^{-1}(\lambda_i)$ is measurable.  For such a measurable step function $\phi$, the sets $A_1,\ldots,A_n$ form a partition of $\pset$, and $\phi=\sum_{i=1}^n \lambda_i\indicator{\mss_i}$; we shall refer to this as the \emph{standard representation} of $\phi$.  The set of $\L$-valued measurable step functions is denote by $\mathcal{S}(\pset,\salg;\ls)$. We define the integral of $\phi$ with respect to $\npm$ by 	
 \[
 \int_S \phi \d \mu \coloneqq \sum_{i=1}^n \lambda_i\npm(\mss_i).
 \]
It is easy to verify that $\mathcal{S}(\pset,\salg;\ls)$ is lattice-ordered algebra and $\L$-module with respect to the pointwise operations and order, and the integral is a positive $\ls$-linear map from $\mathcal{S}(\pset,\salg;\ls)$ into $\ls$.  Indeed, let $\phi=\sum_{i=1}^n \lambda_i\indicator{\mss_i}$ and $\psi=\sum_{j=1}^m \gamma_j\indicator{B_j}$ be two such functions (in their standard representations) and $\lambda\in\L$.  Define $\{\eta_1,\ldots,\eta_p\}\coloneqq \{\lambda \lambda_i+\gamma_j~:~ A_i\cap B_j\neq \emptyset\}$, and for every $k=1,\ldots,p$ let $C_k \coloneqq \bigcup\{A_i\cap B_j ~:~ \lambda \lambda_i+\gamma_j = \eta_k\}$.  Then $C_1,\ldots,C_p$ is a measurable partition of $S$, and
\[
\lambda \phi + \psi = \sum_{k=1}^p \eta_k \indicator{C_k}.
\]
Hence, $\lambda \phi + \psi\in \mathcal{S}(\pset,\salg;\ls)$.  Similar arguments show that $\psi\phi$ and $|\phi|$ are $\L$-valued measurable step functions.  For $\L$-linearity of the integral, we have
\[
\lambda \int_S\phi \d \mu + \int_S \psi \d\mu = \sum_{i=1}^n\sum_{j=1}^m\left(\lambda\lambda_i+\gamma_j\right)\mu(A_i\cap B_j).
\]
Disregarding instances where $A_i\cap B_j=\emptyset$ and rearranging the terms yield
\[
\lambda \int_S\phi \d \mu + \int_S \psi \d\mu = \sum_{k=1}^p \eta_k\mu(C_k) = \int_S \lambda\phi+\psi \d \mu.
\]
We note that from the above, it follows that for not necessarily disjoint $A_1,\ldots,A_n\in\Sigma$ and not necessarily distinct $\lambda_1,...,\lambda_n\in \L$,  $\phi = \sum_{i=1}^n\lambda_i\chi_{A_i}$ is an $\L$-valued measurable step function, and
\[
\int_S \phi \d \mu = \sum_{i=1}^n \lambda_i\npm(\mss_i).
\]
\begin{remark}
We remark that in the classical case of real-valued measures and functions, measurable functions are introduced first, and it is shown that the set of all measurable functions is vector space.  To show that the space of measurable step functions is, for instance, a vector space, it is then sufficient to observe that a linear combination of functions with finite ranges also has a finite range.  In our approach, the definition of a measurable step function precedes logically that of a general measurable functions, and so more direct arguments are necessary to verify that $\mathcal{S}(\pset,\salg;\ls)$ is a lattice ordered algebra.
\end{remark}

\begin{defn}\label{def:positive_extended_valued_measurable_function}
A  function  $f\colon\pset\to \poslsext$  is called \emph{measurable} if there exists a sequence of positive measurable step functions $\seq{\phi }$ such that $\phi_n\uparrow f$ pointwise.
\end{defn}
The set of all measurable functions $f \colon \pset \to \poslsext$ is denoted $\poslmeasfunext$, and $\poslmeasfun$ is its subset consisting of finite-valued positive measurable functions; that is, functions in $\poslmeasfunext$ taking values in $\L^+$.  As in the case of real-valued functions, $\poslmeasfunext$ and $\poslmeasfun$ contain the pointwise sum, supremum, and infimum of any two of their elements. $\poslmeasfunext$ and $\poslmeasfun$ are also closed under pointwise multiplication by elements of $\posls$. We now show that $\poslmeasfunext$ also contains the pointwise supremum of any monotone increasing sequence of functions.  The proof relies on \cite[Lemma 36]{Anguelov_vanderWalt_2005_order_convergence_structure}.  Since we make use of it several times, we state it below.

\begin{proposition}\label{Prop:Diagonal-type increasing sequence}
Let $P$ be a lattice.  For ever $n\in \N$, let $(p_{nm})$ be a sequence in $P$ so that $p_{nm}\uparrow_m p_n$.  Assume further that $p_n\uparrow_n p$.  For every $n\in \N$ define
\[
q_n \coloneqq \bigvee \{p_{1n},p_{2n},\ldots,p_{nn}\}.
\]
Then $q_n \uparrow_n p$.
\end{proposition}

\begin{lemma}\label{res:pointwise sup of increasing sequence of measurable functions is measurable}
Let $(f_n)$ be a monotone increasing sequence in $\poslmeasfunext$, and let $f$ be the pointwise supremum of $(f_n)$. Then $f\in\poslmeasfunext$.
\end{lemma}

\begin{proof}
The proof is similar to the proof of the monotone convergence theorem in \cite[Theorem~6.9]{de_jeu_jiang:2022a}.

For each $n\in\N$,  we take a sequence $(\phi_{n,k})$ in $\mathcal{S}(\pset,\salg;\ls)$ such that $\phi_{n,k} \uparrow f_n$ pointwise in $\poslsext$. We set $\psi_n=\displaystyle\bigvee_{i=1}^n \phi_{i,n}$ for $n\in\N$.
	Then each $\psi_n$ is in $\mathcal{S}(\pset,\salg;\ls)$ and $\psi_n \uparrow f$ pointwise in $\poslsext$ by \Cref{Prop:Diagonal-type increasing sequence}.  Hence, $f$ is in $\poslmeasfunext$.
\end{proof}

For any function $f\colon\pset\to \ls$, its positive and negative  parts are defined in the usual way, i.e. $\pos{f}(s)=f(s)\vee 0$, $\negt{f}(s)=[-f(s)]\vee 0$, for every $s\in\pset$. An $\ls$-valued function  $f$ on $\pset$ is called \emph{measurable} if its positive and negative parts are both in $\poslmeasfun$.  We denote by $\lmeasfun$ the collection of all $\ls$-valued measurable functions.

\begin{proposition}\label{res:measurable_function_space_is_sDc} $\lmeasfun$ is a \sDc\ vector lattice with positive cone $\poslmeasfun$.
\end{proposition}
\begin{proof}
It is easy to see that $\poslmeasfun$ is a cone in $\lmeasfun$, and, by definition of $\poslmeasfun$ it is generating.  As remarked above, the pointwise supremum of any $f,g\in \poslmeasfun$ belong to $\poslmeasfun$; this is clearly the least upper bound of $\{f,g\}$ in $\poslmeasfun$.  Therefore $\lmeasfun$ is a vector lattice.

To show that $\lmeasfun$ is \sDc, we only need to show that it is $\sigma$-monotone complete.  Let $(f_n)$ be a monotone increasing in $\poslmeasfun$, bounded from above by $g\in\poslmeasfun$. We define  $f$ as the pointwise limit of $(f_n)$, i.e. $f(s)=\sup_n f_n(s)$ for all $s\in\pset$. By \Cref{res:pointwise sup of increasing sequence of measurable functions is measurable}, $f\in \poslmeasfunext$.  But $f(s)$ is in $\ls$ since $f_n\leq g$ for all $n$, and so $f\in \poslmeasfun$.
\end{proof}

For a function $f:S\to \lsext$ and $\lambda\in \lsext$, we define $\{f\leq \lambda\}\coloneqq \{s\in S ~:~ f(s)\geq \lambda\}$.  The sets $\{f\geq \lambda\}$ and $\{f=\lambda\}$ are defined analogously.

\begin{proposition}\label{res:inverse images of translated cone and points are measurable}
Let $f\in \poslmeasfunext$ and $\lambda\in \L^+$.  Then each of the sets $\{f\leq \lambda\}$, $\{f\geq \lambda\}$, $\{f= \lambda\}$, and $\{f = \infty\}$ is measurable.
\end{proposition}

\begin{proof}
There exists a sequence $(\phi_n)$ of positive, measurable step functions so that $\phi_n\uparrow f$ pointwise on $S$.  We have
\[
\{f\leq \lambda\} = \bigcap_{n\in\N} \{\phi_n\leq \lambda\}.
\]
For each $n\in \N$, let $\phi_n = \sum_{i=1}^{k_n}\lambda_i^{(n)}\chi_{A_i^{(n)}}$, in standard form.  Then
\[
\{\phi_n\leq \lambda\} = \bigcup\{A_i^{(n)} ~:~ \lambda_i^{(n)}\leq \lambda\}.
\]
Therefore, each $\{\phi_n\leq \lambda\}$ is measurable so that $\{f\leq \lambda\}$ is measurable.

For the proof that $\{f\geq \lambda\}$ is measurable, we first consider the case when $f$ is bounded. Let $\lambda_1 \coloneqq \sup\{f(s) ~:~ s\in S\}$, and define $g\coloneqq \lambda_1-f$.  Then $g\in\poslmeasfun$ and $\{f\geq \lambda\} = \{g\leq \lambda_1-\lambda\}$, which is measurable by (i).  For the general case, let $h(s)\coloneqq f(s)\wedge \lambda$ for each $s\in S$.  Then $h$ is bounded and measurable, and so $\{f\geq \lambda\} = \{h\geq \lambda\}$ is measurable.

Since $\{f=\lambda\}=\{f\leq \lambda\}\cap \{f\geq \lambda\}$ and $\{f=\infty\} = \bigcap_{n\in\N}\{f\geq n\}$, it follows that both these sets are measurable.
\end{proof}

In order to extend the integral to from measurable step functions to all measurable functions, we need to approximate functions in $\poslmeasfunext$ by bounded measurable step functions, which we now define.

\begin{defn}\label{d:bounded_step_functions}
    A measurable step function $\phi=\sum_{i=1}^m \lambda_i\indicator{\mss_i}$ is a \emph{bounded measurable step function} if $\lam_i \in \L_1$ for all $1 \leq i \leq m$. We denote by $\lstepfun$ the set of all bounded measurable step functions.
\end{defn}

Note that $\lstepfun$ is a subalgebra and sublattice of $\mathcal{S}(\pset,\salg;\ls)$, and its positive cone is
$$
\poslstepfun = \set{\phi\in\lstepfun\colon \phi(s)\in\pos{\L_1}, \text{ for all } s\in \pset}.
$$

\begin{thm}\label{t:approximate_by_bounded_step_functions}
Let $f\in \poslmeasfunext$.  Then there exists an increasing sequence $(\varphi_n)$ of bounded measurable step functions so that $\varphi_n \uparrow f$ pointwise.
\end{thm}

\begin{proof}
By definition of $\poslmeasfunext$ there exists a sequence $(\psi_m)$ of positive measurable step functions so that $\psi_m \uparrow f$ pointwise.  For every $m\in\N$ let
\[
\psi_m = \sum_{i=1}^{k_m} \lambda_i^{(m)}\indicator{A_i^{(m)}}.
\]
Fix $m\in \N$. For every $n\in\N$, define $\varphi_{nm}:S\to \L_1$ by $\varphi_{nm}(s) \coloneqq \psi_m(s)\wedge n$, $s\in S$.  We note that
\[
\varphi_{nm} =\sum_{i=1}^{k_m} \left(\lambda_i^{(m)}\wedge n\right) \indicator{A_i^{(m)}},
\]
so that it is a bounded measurable step function.  Since $1$ is a weak unit in $\ls$, it follows that $\varphi_{nm}\uparrow \psi_m$ pointwise.

For every $n\in \N$ define
\[
\varphi_n \coloneqq \bigvee \{ \varphi_{1n},\varphi_{2n},\ldots,\varphi_{nn}\}.
\]
Then $(\varphi_n)$ is a sequence of positive, bounded measurable step functions.  By Proposition \ref{Prop:Diagonal-type increasing sequence}, $\varphi_n\uparrow f$ pointwise.
\end{proof}

Combining \Cref{t:approximate_by_bounded_step_functions} and \Cref{def:positive_extended_valued_measurable_function} yields that a function $f \colon S \to \poslsext$ is in $\poslmeasfunext$ if and only if it is the pointwise supremum of an increasing sequence in $\poslstepfun$.

\section{Proper measures}\label{sec:proper_measures}

The validity of our integration theory depends on the fact that if a sequence of bounded measurable step functions decrease pointwise to 0, then so do the integrals.
\begin{defn}
An $\ls$-valued measure $\mu$ on $(S,\Sigma)$ is \emph{proper} if for every sequence $(\varphi_n)$ in $\poslstepfun$, if $\varphi_n\downarrow \zerofunction$ pointwise, then
\[
\int_S \phi_n  \d \mu \downarrow 0 \text{ in }\L.
\]
\end{defn}
The next example, which is a slight modification of \cite[Example~5.4]{wickstead:1982}, shows that fairly natural $\L$-valued measures need not be proper.
\begin{example}\label{e:non_proper_measure}
    Let $\L$ be such that $L^\infty[0,1) \subseteq \L \subseteq L^0[0,1)$ and let $S = [0,1)$ be equipped with the $\L$-valued measure $\mu$ defined on the Lebesgue $\sigma$-algebra by $\mu(A) = \mathbf{1}_A$.  Here $\mathbf{1}_A\in \L$ denotes the real-valued indicator function of a Lebesgue measurable set $A\subseteq S$. For each $n \in \N$ and $0 \leq k \leq 2^n - 1$, denote by $I_{n,k}$ the interval $[k2^{-n}, (k+1)2^{-n})$, so that for a fixed $n \in \N$, $S$ is the disjoint union $\cup_{k=0}^{2^n-1} I_{n,k}$. In particular, for each $s \in S$, there is a unique $I_{n,k}$ such that $s \in I_{n,k}$.

For $n \in \N$, define $\phi_n \colon S \to \L$ by $\phi_n(s) = \mathbf{1}_{I_{n,k}}$ whenever $s \in I_{n,k}$. In other words,
\[
\phi_n = \sum_{k=0}^{2^n-1} \mathbf{1}_{I_{n,k}} \indicator{I_{n,k}},
\]
so $\phi_n$ is a step function. Now for each $s \in S$, $\phi_n(s)$ is the indicator function of an interval of size $2^{-n}$ which decreases to $0$ in $\L$ (in order), so $(\phi_n)$ decreases to $0$ pointwise. But
\[
\int_S \phi_n \d \mu = \sum_{k=0}^{2^n-1} \mathbf{1}_{I_{n,k}} \mu(I_{n,k}) = \sum_{k=0}^{2^n-1} \mathbf{1}_{I_{n,k}} \mathbf{1}_{I_{n,k}} = \mathbf{1}_{[0,1)},
\]
showing that $\int_S \phi_n \d \mu$ does not decrease to $0$.
\end{example}

We next show that there is a large class of measures which are proper. The terminology of the next definition is taken from \cite[Definition~3.4]{wickstead:1982}.  We remark that the `$\sigma$' indicates `sum', and does not imply any countability assumptions.
\begin{defn}\label{d:sigma-scalar}
An $\ls$-valued measure $\mu$ on $(S,\Sigma)$ is \emph{$\sigma$-scalar} if $\mu \colon \Sigma\to \ls_1$, and there exists a collection $\{m_j\}_{j\in J}$ of positive, real-valued measures on $(S,\Sigma)$ and a collection $\{\lambda_j\}_{j\in J}$ in $\ls_1^+$ so that, for every $A\in \Sigma$,
\[
\mu(A) = \sum_{j\in J} m_j(A)\lambda_j \text{ in } \ls_1.
\]
\end{defn}
We note that the potentially infinite sum in \Cref{d:sigma-scalar} should be interpreted in terms of order convergence, as explained in \Cref{Sec:summability}.

\begin{thm}\label{Thm:  Sigma-scalar measures are proper}
Let $\mu$ be a $\sigma$-scalar measure.  Assume that the order continuous functionals on $\ls_1$ separate the points of $\ls_1$.  Then $\mu$ is proper.
\end{thm}
Note that the above assumption on $\ls_1$ is equivalent to requiring that $K$ in \eqref{e:L_inside_continuous_functions} is hyper-Stonean.
\begin{proof}
Let $\Psi$ be a positive order continuous linear functional on $\ls_1$. Suppose first that $\mu = am$ with $0\leq a\in \ls_1$ and $m$ a real-valued measure on $\Sigma$.  Let $\varphi$ be a positive, measurable $\ls_1$-valued step function on $S$ so that
\[
\varphi = \sum_{i=1}^n\lambda_i \indicator{A_i}.
\]
Define $g^{(\varphi)}:S\to \R$ as
\[
g^{(\varphi)}(s) \coloneqq \Psi(\varphi(s)a),~ s\in S.
\]
Then $g^{(\varphi)}$ is a real-valued measurable step function on $S$.  In particular,
\[
g^{(\varphi)} = \sum_{i=1}^n \Psi(\lambda_i a)\indicator{A_i}.
\]
Observe that
\begin{eqnarray}
\begin{array}{lll}
\displaystyle \int_S g^{(\varphi)} \ d m & = & \displaystyle\sum_{i=1}^n \Psi(\lambda_i a)m(A_i)\bigskip \\
& = & \displaystyle\Psi \left( \sum_{i=1}^n \lambda_i a m(A_i) \right)\bigskip \\
& = & \displaystyle\Psi\left( \int_S \varphi \d \mu\right). \\
\end{array} \label{EQ: L-integral-to-scalar-integral}
\end{eqnarray}

Let $(\varphi_n)$ be a sequence of $\ls_1$-valued, measurable step functions.  Assume that $\varphi_n\downarrow 0$ pointwise on $S$.  Then $g^{(\varphi_n)}\downarrow 0$ pointwise on $S$.  By \eqref{EQ: L-integral-to-scalar-integral} and the classical Lebesgue Monotone Convergence Theorem,
\[
\Psi\left(\int_S \varphi_n \d \mu \right) \downarrow 0.
\]
Since this is true for every positive order continuous functional on $\ls_1$, and these functionals separate the points of $\ls_1$, it follows that
\[
\int_S \varphi_n \ d\mu \downarrow 0.
\]
Indeed, suppose that $0\leq \lambda \leq \int_S \varphi_n \d\mu$ for all $n\in\N$.  Then for every positive, order continuous functional $\Psi$ on $\ls_1$ we have
\[
0\leq \Psi(\lambda)\leq \Psi\left(\int_S \varphi_n \ d\mu\right)=0
\]
so that $\Psi(\lambda)=0$. The fact that the order continuous functionals on $\ls_1$ separates the points of $\ls_1$ implies that $\lambda=0$.

Now consider the general case:  $\mu = \sum_{j\in J} a_j m_j$.  For each $j\in J$ let $\mu_j \coloneq a_j m_j$.  Let $(\varphi_n)$ be a sequence of $\ls_1$-valued, measurable step functions so that $\varphi_n \downarrow 0$ pointwise on $S$.

For every $n\in\N$ let
\[
\varphi_n = \sum_{i=1}^{k_n} \lambda_i^{(n)}\indicator{A_i^{(n)}}.
\]
Then
\[
\int_S \varphi_n \ d\mu = \sum_{i=1}^{k_n} \lambda_i^{(n)}\mu(A_{i}^{(n)}) = \sum_{i=1}^{k_n} \lambda_i^{(n)}\sum_{j\in J}\mu_j(A_{i}^{(n)}).
\]
By the order continuity of multiplication in $\ls$,
\[
\int_S \varphi_n \ d\mu = \sum_{i=1}^{k_n} \sum_{j\in J}\mu_j(A_{i}^{(n)}) \lambda_i^{(n)}.
\]
It now follows from \cite[Theorem 15.8 (iii)]{luxemburg_zaanen_RIESZ_SPACES_VOLUME_I:1971} that
\[
\int_S \varphi_n \ d\mu = \sum_{j\in J} \sum_{i=1}^{k_n} \mu_j(A_{i}^{(n)}) \lambda_i^{(n)} = \sum_{j\in J} \int_S \varphi_n \d \mu_j.
\]
By the linearity and order continuity of $\Psi$,
\[
\Psi\left( \int_S \varphi_n \d \mu \right) = \sum_{j\in J} \Psi\left(\int_S \varphi_n \d \mu_j\right).
\]
The monotonicity of the integral (see \Cref{sec:measurable_functions_and_measurable_step_functions}) implies that $\left( \int_S \varphi_n \d \mu \right)$ is a decreasing sequence in $\L^+$.  Therefore, by the order continuity of $\Psi$,
\[
\Psi\left( \bigwedge_{n\in\N}  \int_S \varphi_n \d \mu \right) =      \bigwedge_{n\in\N} \Psi\left(\int_S \varphi_n \ d \mu \right) = \bigwedge_{n\in\N}\sum_{j\in J} \Psi\left(\int_S \varphi_n \d \mu_j\right).
\]
For every $n\in \N$, define $g^{(\Psi)}_n:J\to \R$ as
\[
g^{(\Psi)}_n(j) \coloneqq \Psi\left(\int_S \varphi_n \ d \mu_j\right),~ j\in J.
\]
Then for all $n\in \N$,
\[
\Psi\left( \bigwedge_{n\in\N}  \int_S \varphi_n \d \mu \right) = \bigwedge_{n\in\N} \int_J g^{(\Psi)}_n(j) \d j
\]
where $\d j$ denotes the counting measure on the powerset of $J$.  We have already shown that $\displaystyle \int_S \varphi_n \d \mu_j\downarrow 0$ for every $j\in J$. Hence, once again by the order continuity of $\Psi$, $g^{(\Psi)}_n(j)\downarrow 0$ for every $j\in J$.  Therefore, applying the classical Lebesgue Monotone Convergence Theorem to the counting measure on $J$ yields
\[
\Psi\left( \bigwedge_{n\in\N}  \int_S \varphi_n \d \mu \right) = \bigwedge_{n\in\N} \int_J g^{(\Psi)}_n(j) \d j = 0.
\]
Since this holds for all positive, order continuous functionals on $\ls_1$, and these functionals separate the points of $\ls_1$, it follows that $\displaystyle\int_S \varphi_n \d \mu \downarrow 0$.
\end{proof}

\section{Integration of measurable functions}\label{sec:integration}

In the rest of the paper, we fix a measure space $\msm$ such that $\mu$ is proper. The next lemma is crucial for proving well-definedness of the integral of measurable functions.
 \begin{lemma}\label{res:key_to_well_definedness_of_L-integral}
	Let $\phi$ and $\seq{\phi}$ be in $\poslstepfun$ such that $\phi_n$ is increasing and $\phi(s)\leq\sup_{n\geq 1}\phi_n(s)$ in $\poslsext$ for every $s\in S$. Then
	\[
 \int \phi \d\mu \leq \bigvee_{n=1}^\infty \int_S \phi_n \d\mu.
 \]
\end{lemma}

\begin{proof}
The inequality clearly holds when $\dis{\bigvee_{n=1}^\infty \int_S \phi_n \d\mu = \infty}$, so we may assume that $\dis{\bigvee_{n=1}^\infty \int_S \phi_n \d\mu \in \pos{\ls}}$. For each $s \in S$, the assumptions imply that $(\phi(s) - \phi_n(s)) \vee 0$ decreases to $0$. Hence $([\phi - \phi_n] \vee 0)$ is a sequence in $\poslstepfun$ decreasing pointwise to $0$. Since $\mu$ is proper, it follows that
\[
\bigwedge_{n=1}^\infty \int_S (\phi - \phi_n) \vee 0 \d\mu = 0.
\]
The integral is a positive operator from the vector lattice $\poslstepfun$ to $\L$, so
\[
\int_S \phi \d \mu - \bigvee_{n=1}^\infty \int_S \phi_n \d\mu = \bigwedge_{n=1}^\infty \int_S (\phi - \phi_n) \d\mu \leq  \bigwedge_{n=1}^\infty \int_S (\phi - \phi_n) \vee 0 \d\mu = 0,
\]
which proves the lemma.
\end{proof}

 We now extend the integral to $\poslmeasfunext$ in the obvious way: for any $f\in\poslmeasfunext$, we define the integral of $f$ with respect to $\npm$ as
\begin{equation}\label{eq:integral_of_positive_measurable_function}
	 \int_S f \d\mu \coloneqq \bigvee_{n=1}^\infty \int \phi_n \d \mu,
\end{equation}	
where $\seq{\phi}$ is a sequence in $\poslstepfun$ such that $\phi_n\uparrow f$ pointwise. If the supremum exists in $\ls$, then we say that $f$ is \emph{integrable}. In order to show that this definition is independent of the choice of the sequence $\seq{\varphi}$, let $\seq{\psi} \subseteq \poslstepfun$ be a second sequence such that $\psi_n\uparrow f$ pointwise. Since for every $k\in \N$, $\psi_k(s)\leq f(s)=\sup_{n\geq 1}\phi_n(s)$ for all $s\in S$, by \Cref{res:key_to_well_definedness_of_L-integral} we have $\dis{\int_S \psi_k  \d\mu \leq \sup_{n\geq 1}\int_S \phi_n \d\mu}$. And since this is true for all $k$, we obtain $\dis{\sup_{k\geq 1} \int_S \psi_k \d\mu \leq \sup_{n\geq 1} \int_S \phi_n \d\mu}$. The reverse inequality holds by the same argument, and we conclude that $\dis{\int_S f \d\mu}$ is well defined as an element of $\poslsext$.

The following basic properties follow as in the classical case, i.e. when $\L=\R$.

\begin{lemma}\label{res:integral_on_positive_functions_has_usual_properties}
Let $f_1,\,f_2\colon \pset\to\poslsext$ be  measurable functions, and let $\lambda_1,\,\lambda_2\in\posls$  Then
	\begin{enumerate}[(1)]
		\item $\int_S (\lambda_1 f_1 + \lambda_2f_2) \d \mu = \lambda_1 \int_S f_1 \d \mu +\lambda_2 \int_S f_2 \d \mu$ in $\poslsext$;\smallskip \label{part:integral_on_positive_functions_has_usual_properties_1}
		\item If $f_1\leq f_2$ pointwise in $\poslsext$, then $\int_S f_1 \d \mu \leq \int_S f_2 \d \mu$ in $\poslsext$.\label{part:integral_on_positive_functions_has_usual_properties_2}
	\end{enumerate}
\end{lemma}
\begin{proof}
	The result follows by the $\ls$-linearity of the integral on $\lstepfun$ and  $\ls$-linearity of order convergence.
\end{proof}
We note that $\L^+$ is a sublattice of $\poslsext$, so \Cref{res:integral_on_positive_functions_has_usual_properties} applies in the case of $\L^+$-valued functions and integrals.

\begin{lemma}\label{res:positive_integrable_is_almost_everywhere_finite} Let $f\in\poslmeasfunext$.  If $\dis{\int_S f \d \mu \in \posls}$, then $f$ is almost everywhere finite.
\end{lemma}
\begin{proof}
     According to \Cref{res:inverse images of translated cone and points are measurable} $A\coloneqq \{f=\infty\}$ is measurable. Then $0 \leq n \indicator{A} \leq f$ pointwise for each $n \in \N$ and integration yields $\dis{0 \leq n \mu(A) \leq \int_S f \d \mu \in \L^+}$. Since $\L$ is Archimedean, $\mu(A) = 0$.
\end{proof}

\begin{lemma}\label{res:almost_everywhere_zero_function_has_zero_integral}
Let $f\in\poslmeasfunext$. If $f(s)=0$ for almost all $s\in\pset$, then $\dis{\int_S f \d\mu = 0}$.
\end{lemma}

\begin{proof}
Suppose that $f(s)=0$ for almost all $s\in\pset$.  Let $\varphi\in \poslstepfun$ such that $\varphi\leq f$.  Express $\varphi$ in standard form, $\varphi = \sum_{i=1}^n\lambda_i\chi_{A_i}$.  For every $i=1,\ldots,n$, if $\lambda_i\neq 0$ then $\mu(A_i)=0$.  Therefore $\int_S\varphi\d \mu=0$.  The result now follows from the definition of the integral.
\end{proof}

\begin{remark}\label{res:pnorms_on_Lp_is_seminorm}
	In the real-valued functions case, a positive measurable function $f$ is almost everywhere zero if and only if $\displaystyle\int_S f \d \mu=0$, see \cite[Lemma~6.5]{de_jeu_jiang:2022a}. For  $\ls$-valued functions, \Cref{res:almost_everywhere_zero_function_has_zero_integral} states that if  $f\in\poslmeasfunext$ is almost everywhere zero, then $\displaystyle\int_S f \d\mu=0$. But the converse need not hold. Indeed, consider a step function $\phi=\sum_{i=1}^n\lambda_i\indicator{\mss_i}$, then $\displaystyle\int_S \abs{\phi} \d\mu = 0$ implies  $\abs{\lambda_i}\npm(\mss_i)=0$ for each $i$, which holds if and only if $\lambda_i$ and $\npm(\mss_i)$ are disjoint in $\L$ for all $i$.
\end{remark}

Combining \Cref{res:positive_integrable_is_almost_everywhere_finite,res:almost_everywhere_zero_function_has_zero_integral} we see that a for function $f\in \poslmeasfunext$, $\int_S f\d\mu\in\ls$ if and only if there exists a function $\hat{f}\in \poslmeasfun$ such that $f=\hat{f}$ almost everywhere and $\int_S f\d\mu =\int_S \hat{f}\d\mu$.  Therefore, when defining integrable functions, we restrict ourselves to $\ls$-valued measurable functions.

An $\ls$-valued measurable function $f$ is called \emph{integrable} with respect to $\npm$, if $\int_S \abs{f} \d\mu \in\ls$.  By \Cref{res:integral_on_positive_functions_has_usual_properties}, $f$ is integrable if and only if its  positive and negative parts are both integrable.   For any $f\in\integrablelfun$, we define its integral with respect to $\npm$ as
$$
\int_S f \d \mu \coloneqq \int_S \pos{f} \d \mu - \int_S \negt{f} \d \mu.
$$
We denote the set of all positive integrable function as $\posintegrablelfun$, and  the set of all integrable functions as $\integrablelfun$.
\begin{proposition}\label{res:L1_is_ideal_of_measurable_function_space}
$\integrablelfun$ is an order ideal of $\lmeasfun$, so it inherits the structure of a \sDc~ vector lattice from $\lmeasfun$. The integral is a positive, $\L$-linear operator from $\integrablelfun$ to $\L$.
\end{proposition}
\begin{proof}
That $\integrablelfun$ is an order ideal of $\lmeasfun$ follows immediately from \Cref{res:integral_on_positive_functions_has_usual_properties}.  Invoking \Cref{res:integral_on_positive_functions_has_usual_properties}, the usual arguments show that the integral is a positive $\L$-linear operator.
\end{proof}

\section{Properties of the integral}\label{sec:properties_of_the_intergal}

\begin{thm}[Monotone convergence theorem]\label{res:MCT_increasing} Suppose $(f_n)$ is a monotone increasing sequence in $\poslmeasfunext$, and let $f$ be the pointwise supremum of $(f_n)$. Then $f\in\poslmeasfunext$ and
$$
\int_S f \d \mu = \bigvee_{n=1}^\infty \int_S f_n \d\mu.$$
Therefore, if $\dis{\bigvee_{n=1}^\infty \int_S f_n \d\mu \in \ls}$, then $f$ is integrable, so it is almost everywhere finite.
	\end{thm}

The proof of this result is similar to the proof of the monotone convergence theorem in \cite[Theorem~6.9]{de_jeu_jiang:2022a}.

\begin{proof}
It follows from \Cref{res:pointwise sup of increasing sequence of measurable functions is measurable} that $f\in \poslmeasfunext$.  In the proof of that result the following is established:  There exists a sequence $(\psi_n)$ in $\lstepfun$ so that $\psi_n\uparrow f$ pointwise, and $\psi_n\leq f_n$ for every $n\in \N$.  It follows that $\int_S \psi_n \d\mu \leq \int_S f_n \d\mu$ for all $n\in\N$, and so
 \[
 \int_S f \d\mu = \bigvee_{n=1}^\infty \int_S \psi_n \d\mu \leq \bigvee_{n=1}^\infty \int_S f_n \d\mu.
 \]
 The positivity of the integral implies that $\dis{\int_S f \d\mu \geq \bigvee_{n=1}^\infty \int_S f_n \d\mu}$. We conclude that
 $$
 \int_S f \d\mu = \bigvee_{n=1}^\infty \int_S f_n \d\mu.\qedhere
 $$
\end{proof}

The straightforward proof of the following result is similar to the proof of \cite[Corollary~6.10]{de_jeu_jiang:2022a}.

\begin{corollary}\label{res:MCT_decreasing} Let $(f_n)$ be a decreasing sequence in $\poslmeasfunext$ and $f\in \poslmeasfunext$ be such that  $f_n(s)\downarrow f(s)$ in $\poslsext$ for almost all $s\in \pset$. If $\dis{\int_S f_1 \d\mu \in\ls}$, then $$\int_S f \d\mu = \bigwedge_{n=1}^\infty \int_S f_n \d\mu.$$
\end{corollary}

\begin{proof}
Redefining the $f_n$ and $f$ to be zero on a suitable set of measure $0$ if necessary, we may suppose that $f_n\downarrow f$ pointwise on $S$, see \Cref{res:almost_everywhere_zero_function_has_zero_integral}.  Then $f_1-f_n\uparrow f_1-f$ pointwise on $S$, and so by \Cref{res:MCT_increasing}, $\displaystyle \int_S (f_1-f_n)\d \mu \uparrow \int_S (f_1-f) \d\mu$.  Since $0\leq f\leq f_n\leq f_1$ for all $n\in \N$ and $\displaystyle \int_S f_1 \d\mu \in \L$, it follows from the monotonicity of the integral that $\displaystyle \int_S f \d\mu$ and the $\displaystyle \int_S f_n \d\mu$ also belong to $\L$.  Therefore, by \Cref{res:L1_is_ideal_of_measurable_function_space},
$$
\int_S f_1 \d\mu - \int_S f_n \d\mu \uparrow \int_S f_1 \d\mu - \int_S f \d\mu
$$
which gives $\displaystyle \int_S f_n \d\mu \downarrow \int_S f\d\mu$.
\end{proof}

\begin{lemma}[Fatou's lemma]\label{res:Fatou's_lemma}
Let $(f_n)$ be sequence in $\poslmeasfun$.  Then
\[\int_S\bigvee_{n=1}^\infty \bigwedge_{k=n}^{\infty} f_k \d\mu \leq \bigvee_{n=1}^\infty \bigwedge_{k=n}^{\infty}\int_S f_k \d\mu .\]
\end{lemma}

The proof for this lemma is similar to \cite[Theorem~6.12]{de_jeu_jiang:2022a}.

\begin{proof}
For any $n\geq 1$, the set $\set{f_k}_{k=n}^\infty$
is bounded from below by $0$. By the Dedekind $\sigma$-completeness of $\lmeasfun$, $\displaystyle\bigwedge_{k=n}^{\infty} f_k$ is measurable. Since $0\leq \displaystyle\bigwedge_{k=n}^{\infty} f_k\leq f_j$
 for all $n\in\N$ and $j\geq n$, by part \partref{part:integral_on_positive_functions_has_usual_properties_2} of \Cref{res:integral_on_positive_functions_has_usual_properties},
$$
 \int_S \bigwedge_{k=n}^{\infty} f_k \d\mu \leq  \int_S f_j \d\mu,
$$
for every $n\in\N$ and $j\geq n$. Therefore
\begin{equation}
\int_S \bigwedge_{k=n}^{\infty} f_k \d\mu \leq \bigwedge_{k=n}^\infty \int_S f_k \d\mu\label{eq:Fatou's_lemma_1}
\end{equation}
for every $n\in\N$. Applying \Cref{res:MCT_increasing} to the monotone increasing sequence $\displaystyle(\bigwedge_{k=n}^{\infty} f_k)$, we obtain
\begin{equation}
\int_S\bigvee_{n=1}^\infty \bigwedge_{k=n}^{\infty} f_k \d\mu = \bigvee_{n=1}^\infty \int_S \bigwedge_{k=n}^\infty f_k \d\mu.\label{eq:Fatou's_lemma_2}
\end{equation}
Combining \eqref{eq:Fatou's_lemma_1} and \eqref{eq:Fatou's_lemma_2} completes the proof.
\end{proof}

\begin{thm}[Dominated convergence theorem]\label{res:DCT} Let $(f_n)$ and $f$ be $\ls$-valued measurable functions on $\pset$ and suppose that $f_n\to f$ pointwise on $S$. If there exists $g\in\integrablelfun$ such that $\abs{f_n}\leq g$ for all $n\in\N$, then
\begin{itemize}
	\item[(1)]$f$ and the $f_n$ are all integrable functions;
	\item[(2)]$\displaystyle{\bigwedge_{n=1}^\infty \bigvee_{k=n}^\infty \int_S \abs{f_k-f} \d\mu = 0}$;
	\item[(3)]$\dis{\int_S f \d\mu = \bigvee_{n=1}^\infty \bigwedge_{k=n}^\infty \int_S f_k \d\mu = \bigwedge_{n=1}^\infty \bigvee_{k=n}^\infty \int_S f_k \d\mu}$.
\end{itemize}
\end{thm}

	The proof for this theorem  is similar to \cite[Theorem~6.13]{de_jeu_jiang:2022a}.

\begin{proof}
Since $f_n\to f$ pointwise and $\abs{f_n}\leq g$, it follows that $\abs{f}\leq g$.
	Now the positivity of the integral implies $(1)$.
	
	Since $2g\geq 2g-\abs{f_k-f}\geq 2g-\abs{f_k}-\abs{f}\geq 0$ for every $k\in\N$, it follows that $(2g -\abs{f_k-f})$ is a sequence in $\integrablelfun$ bounded from below by $0$. So for any $n$,   $\displaystyle{h_n \coloneqq \bigwedge_{k=n}^\infty(2g-\abs{f_k-f})}$ exists in $\pos{\integrablelfun}$, see \Cref{res:L1_is_ideal_of_measurable_function_space}.
	  Clearly $(h_n)$ is a monotone increasing sequence in $\integrablelfun$ and
	$h_n\uparrow 2g$ pointwise. So by  \Cref{res:MCT_increasing} and \Cref{res:MCT_decreasing},
\begin{align*}
	\int_S 2g \d\mu &= \bigvee_{n=1}^\infty \int_S h_n \d\mu \\
	&= \bigvee_{n=1}^\infty \int_S \bigwedge_{k=n}^\infty(2g-\abs{f_k-f}) \d\mu \\
	&\leq \bigvee_{n=1}^\infty \bigwedge_{k=n}^\infty \int_S 2g-\abs{f_k-f} \d\mu \\
	&=\int_S 2g \d\mu - \bigwedge_{n=1}^\infty \bigvee_{k=n}^\infty \int_S \abs{f_k-f} \d\mu.
\end{align*}
By subtracting $\dis{\int_S 2g \d\mu}$ on both side of the above inequatity, we get
$$
\bigwedge_{n=1}^\infty \bigvee_{k=n}^\infty \int_S \abs{f_k-f} \d\mu \leq 0.$$
The reverse inequality is obvious, so the proof of part $(2)$ is complete.

Since $g+f_n\geq 0$ for all $n\in\N$, Fatou's lemma shows that
\begin{align*}
\int_S (g+f) \d\mu &= \int_S \bigvee_{n=1}^\infty \bigwedge_{k=n}^\infty (g+f_n) \d\mu \\
	&\leq \bigvee_{n=1}^\infty \bigwedge_{k=n}^\infty \int_S (g+f_n) \d\mu\\
	&=\int_S g \d\mu + \bigvee_{n=1}^\infty \bigwedge_{k=n}^\infty \int_S f_n \d\mu,
\end{align*}
from which we see that
\begin{equation}\label{1_eq:dominated_convergence_1}
	\int_S f \d\mu \leq \bigvee_{n=1}^\infty \bigwedge_{k=n}^\infty \int_S f_n \d\mu.
\end{equation}
Since $g-f_n\geq 0$ for all $n\geq 1$, Fatou's lemma shows that
\begin{align*}
\int_S (g-f) \d\mu &= \int_S \bigvee_{n=1}^\infty \bigwedge_{k=n}^\infty (g-f_n) \d\mu \\
	&\leq \bigvee_{n=1}^\infty \bigwedge_{k=n}^\infty \int_S (g-f_n) \d\mu\\
	&=\int_S g \d\mu - \bigwedge_{n=1}^\infty \bigvee_{k=n}^\infty \int_S f_n \d\mu,
\end{align*}
from which we see that
\begin{equation}\label{1_eq:dominated_convergence_2}
	\int_S f \d\mu \leq \bigwedge_{n=1}^\infty \bigvee_{k=n}^\infty \int_S f_n \d\mu.
\end{equation}
Combining \eqref{1_eq:dominated_convergence_1} and  \eqref{1_eq:dominated_convergence_2} we have
\[
\bigwedge_{n=1}^\infty \bigvee_{k=n}^\infty \int_S f_n \d\mu \leq \int_S f \leq \bigvee_{n=1}^\infty \bigwedge_{k=n}^\infty \int_S f_n \d\mu \leq \bigwedge_{n=1}^\infty \bigvee_{k=n}^\infty \int_S f_n \d\mu,
\]
which completes the proof of $(3)$.
\end{proof}

\section{The $\ls$-module $\pintegrablelfun$}\label{sec:lp-spaces}
	An $\ls$-valued measurable function is called \emph{essentially bounded} if there exists $\lambda\in\pos{\ls}$ such that $\abs{f}\leq \lambda$ $\npm$-almost everywhere, i.e. the complement of $\set{s\colon \abs{f(s)}\leq \lambda}$ is contained in a measurable set with measure zero. Let $\eblfun$ denote the collection of all essentially bounded functions. For any essentially bounded function $f$ we define $\enorm{f} \coloneqq \inf\set{\lambda\colon \abs{f}\leq \lambda \ \npm\textup{-almost everywhere} }$.  The space $\eblfun$ is clearly an order ideal in $\lmeasfun$, and so a Dedekind $\sigma$-complete vector lattice.

Let $1\leq p < \infty$.  Recall from Section \ref{sec:preliminaries} that $\lambda^p$ is defined in $\ls$ for all $\lambda\in\ls$.   An $\ls$-valued measurable function $f$ is called \emph{$p$-integrable} if $\abs{f}^p\in\integrablelfun$; we denote $\pintegrablelfun\coloneqq\set{f: f\ \textup{is}\  p\textup{-integrable}}$. Since $(|f| \vee |g|)^p = |f|^p \vee |g|^p$, similarly to \Cref{res:L1_is_ideal_of_measurable_function_space}, $\pintegrablelfun$ is also an order ideal of $\lmeasfun$ and hence a Dedekind $\sigma$-complete vector lattice. As in the classical case, we introduce an $\ls$-seminorm on $\pintegrablelfun$ by $$\dis{\pnorm{f}\coloneqq\left(\int_S \abs{f}^p \d\mu \right)^{\frac{1}{p}}},$$ for any $f\in\pintegrablelfun$.  The following result shows that each $\pnorm{\cdot}$ is indeed an $\ls$-seminorm.

\begin{lemma}\label{res:step_function_is_in_Lp}
	For any $p\in[1,\infty]$, $\lstepfun\subseteq\pintegrablelfun$. Furthermore, for any $f$ and $g$ in $\lstepfun$, and for any $\lambda\in\ls$,
	\begin{enumerate}[(1)]
		\item \label{part:step_function_is_in_Lp_1} $\pnorm{f+g}\leq\pnorm{f}+\pnorm{g}$;
		\item \label{part:step_function_is_in_Lp_2} $\pnorm{\lambda f}=\abs{\lambda}\pnorm{f} $;
		\item \label{part:step_function_is_in_Lp_3}$\norm{fg}_1\leq\pnorm{f}\qnorm{g}$ whenever $\frac{1}{p}+\frac{1}{q}=1$, with $\frac{1}{\infty}$ defined as $0$.
	\end{enumerate}
	\end{lemma}

\begin{proof}
Since measurable step functions are bounded, hence essentially bounded,  \partref{part:step_function_is_in_Lp_1} and \partref{part:step_function_is_in_Lp_2} for $p=\infty$ are direct consequences of the properties of absolute value in vector lattices.

If $p\in[1,\infty)$, then for any measurable step function $f$, $\int_S \abs{f}^p \d\mu$ is a finite sum in $\ls$, hence $f\in\pintegrablelfun$.

Take any measurable step functions $f$ and $g$. Then there is a measurable partition $\{A_1,\ldots,A_n\}$ of $S$ so that \[f=\sum_{i=1}^n\alpha_i\indicator{\mss_i},\quad g=\sum_{i=1}^n\beta_i\indicator{\mss_i}\]
for some $\alpha_i, \beta_i$ in $\ls$.
Then, using \cite[Theorem~4.0.4]{L-functional_analysis} in the fourth step, we get
\begin{align*}
	\pnorm{f+g}&=\left[\sum_{i=1}^n\abs{\alpha_i+\beta_i}\npm(\mss_i)\right]^{\frac{1}{p}}\\
	&=\left[\sum_{i=1}^n\abs{\alpha_i+\beta_i}\npm(\mss_i)^{\frac{1}{p} \cdot p}\right]^{\frac{1}{p}}\\
	&=\left[\sum_{i=1}^n\left(\abs{\alpha_i+\beta_i}\npm(\mss_i)^{\frac{1}{p} }\right)^p\right]^{\frac{1}{p}}\\
	&\leq \left[\sum_{i=1}^n\left(\abs{\alpha_i}\npm(\mss_i)^{\frac{1}{p} }\right)^p\right]^{\frac{1}{p}}+\left[\sum_{i=1}^n\left(\abs{\beta_i}\npm(\mss_i)^{\frac{1}{p} }\right)^p\right]^{\frac{1}{p}}\\
	&=\left[\sum_{i=1}^n\abs{\alpha_i}^p\npm(\mss_i)\right]^{\frac{1}{p}}+\left[\sum_{i=1}^n\abs{\beta_i}^p\npm(\mss_i)\right]^{\frac{1}{p}}\\
	&=\pnorm{f}+\pnorm{g}.
	\end{align*}
We also obtain
\[
	\pnorm{\lambda f}=\left[\sum_{i=1}^n\abs{\lambda\alpha_i}^p\npm(\mss_i)\right]^{\frac{1}{p}}=\abs{\lambda}\left[\sum_{i=1}^n\abs{\alpha_i}^p\npm(\mss_i)\right]^{\frac{1}{p}}=\abs{\lambda}\pnorm{f}.
\]
Using \cite[Theorem~4.0.3]{L-functional_analysis} in the third step, we get
\begin{align*}
	\norm{fg}_1&=\sum_{i=1}^n\abs{\alpha_i \beta_i}\npm(\mss_i)\\
	&=\sum_{i=1}^n\abs{\alpha_i \beta_i}\npm(\mss_i)^{\frac{1}{p}+\frac{1}{q}}\\
	&=\sum_{i=1}^n\abs{\alpha_i}\npm(\mss_i)^{\frac{1}{p}} \abs{\beta_i}\npm(\mss_i)^{\frac{1}{q}}\\
	&\leq\left[\sum_{i=1}^n\abs{\alpha_i\npm(\mss_i)^{\frac{1}{p}}}^p\right]^{\frac{1}{p}} \left[\sum_{i=1}^n\abs{\beta_i\npm(\mss_i)^{\frac{1}{q}}}^q\right]^{\frac{1}{q}}\\
	&\leq\left[\sum_{i=1}^n\abs{\alpha_i}^p\npm(\mss_i)\right]^{\frac{1}{p}} \left[\sum_{i=1}^n\abs{\beta_i}^q\npm(\mss_i)\right]^{\frac{1}{q}}\\
	&=\pnorm{f}\qnorm{g}. \qedhere
\end{align*}
\end{proof}

\begin{lemma}\label{res:continuity_of_p-norm_1}
For any sequence $(f_n)$ of positive functions in $\pintegrablelfun$ and every $f \in \pintegrablelfun$, if  $f_n \uparrow f$ pointwise, then $\dis{\int_S f_n^p \d\mu \uparrow \int_S f^p \d\mu}$, so that $\pnorm{f_n} \uparrow \pnorm{f}$.
\end{lemma}
\begin{proof}
If  $f_n\uparrow f$ pointwise, then by \cite[Lemma~4.0.2(ii)]{L-functional_analysis}  $f_n^p\uparrow f^p$ pointwise. The monotone convergence theorem, \Cref{res:MCT_increasing}, implies that $\dis{\int_S f_n^p \d\mu \uparrow \int_S f^p \d\mu}$.
Again by \cite[Lemma~4.0.2(ii)]{L-functional_analysis}, we have $\pnorm{f_n}\uparrow\pnorm{f}$.
\end{proof}
The results of \Cref{res:step_function_is_in_Lp} can now be extended to $\pintegrablelfun$ by the order continuity of the $f$-algebra operations and \Cref{res:continuity_of_p-norm_1}.
\begin{proposition}\label{res:inequality_in_L-integral}
 Suppose $f$ and $g$ are $\ls$-valued measurable functions, and $\lambda\in\ls$.  Then for any $p\geq 1$,
 	\begin{enumerate}[(1)]
 	\item \label{part:inequality_in_L-integral_1}  if $f,g \in \pintegrablelfun$, then $f+g \in \pintegrablelfun$, and $$\pnorm{f+g}\leq\pnorm{f}+\pnorm{g} ;$$
 	\item \label{part:inequality_in_L-integral_2} if $f \in \pintegrablelfun$, then $\lambda f \in \pintegrablelfun$, and $$\pnorm{\lambda f}=\abs{\lambda}\pnorm{f} ;$$
 	\item \label{part:inequality_in_L-integral_3} if $f \in \pintegrablelfun$, $g \in \qintegrablelfun$, and $\frac{1}{p}+\frac{1}{q}=1$, then $fg \in \integrablelfun$ and $$\norm{fg}_1\leq\pnorm{f}\qnorm{g}.$$
 \end{enumerate}
\end{proposition}

For any $1\leq p\leq \infty$, \Cref{res:inequality_in_L-integral} shows that $\norm{\cdot}_p$ is an $\L$-valued seminorm on $\pintegrablelfun$. We let $\kerp$ be the kernel of $\pnorm{\ }$, i.e. $\kerp=\set{f\in\pintegrablelfun\colon \pnorm{f}=0}$. Obviously, $\aezero=\set{f\colon \abs{f}\ \textup{ is zero}\ \npm\textup{-almost everywhere}}$.
\begin{thm}\label{res:embedding_relation_between_Lp_spaces} Let $1\leq p < q \leq \infty$. Then
	\begin{enumerate}[(1)]
	\item \label{part:embedding_relation_between_Lp_spaces_1} $\qintegrablelfun\subseteq\pintegrablelfun$;
	\item \label{part:embedding_relation_between_Lp_spaces_2} $\aezero\subseteq\kerone=\kerp$.
	\end{enumerate}
\end{thm}

\begin{proof}

Take any $p, q\in [1,\infty]$ satisfying  $1 \leq p<q\leq \infty$.
Suppose $f\in\pintegrablelfun$ so that $\abs{f}^p\in\integrablelfun$.  Note that $\frac{p}{q}+\frac{q-p}{q}=1$, so by using H\"{o}lder's inequality (\partref{part:inequality_in_L-integral_3} in \Cref{res:inequality_in_L-integral}) with $\abs{f}^p$ and the constant $1$ function $\chi_S$, we get
\begin{align*}
	\pnorm{f}^p=\norm{\abs{f}^p}_1&=\norm{\ \abs{f}^p\cdot \indicator{\pset}\ }_1\\
	&\leq\norm{\abs{f}^p}_{\frac{q}{p}}\cdot\norm{\indicator{\pset}}_{\frac{q}{q-p}}\\
	&=\qnorm{f}^p\npm(\pset)^{\frac{q-p}{q}}.
\end{align*}
We obtain
$$
\pnorm{f}\leq\qnorm{f}\cdot\npm(\pset)^{\frac{q-p}{qp}}.
$$
Therefore
\begin{equation}\label{eq:embedding_relation_between_Lp_spaces_1} \qintegrablelfun\subseteq\pintegrablelfun\end{equation}
and
\begin{equation}\label{eq:embedding_relation_between_Lp_spaces_2}
\kerq\subseteq\kerp.
\end{equation}

Next, we take $h\in\kerone$ and $1 \leq p < \infty$.
 If  $h$ is a measurable step function $h=\sum_{i=1}^n\lambda_i\indicator{\mss_i}$, then $\dis{\int_S \abs{h} \d\mu = 0}$ implies that $\abs{\lambda_i}\npm(\mss_i)=0$ for all $i=1,2,\dots, n$.  Hence $\abs{\lambda_i}^p\npm(\mss_i)=\abs{\lambda_i}^{p-1}\cdot\abs{\lambda_i}\npm(\mss_i)=0$ for all $i=1,2,\dots, n$. Therefore, $\qnorm{h}=\sum_{i=1}^n\abs{\lambda_i}^q\npm(\mss_i)=0$.

For the general case, we approximate $|h|$ by measurable step functions; then \Cref{res:continuity_of_p-norm_1} yields $\pnorm{h}=0$. Hence
\begin{equation}\label{eq:embedding_relation_between_Lp_spaces_3} \kerone\subseteq\kerp.
\end{equation}
Combining \eqref{eq:embedding_relation_between_Lp_spaces_2} and \eqref{eq:embedding_relation_between_Lp_spaces_3} yields
\[
\kerp=\kerone.
\]
The final statement $\aezero\subseteq\kerone$ is a direct consequence of \Cref{res:almost_everywhere_zero_function_has_zero_integral}.
\end{proof}

By \Cref{res:embedding_relation_between_Lp_spaces}  $\kerone=\kerp$ for any $1\leq p<\infty$ and we denote this set by $\kerall$. It follows from the monotonicity of the integral and the Monotone Convergence Theorem (\Cref{res:MCT_increasing}) that $\kerall$ is a $\sigma$-ideal in $\pintegrablelfun$ and that $\aezero$ is a $\sigma$-ideal in $\eblfun$.  Recall that an ideal in a vector lattice is a $\sigma$-ideal if it is closed with respect to suprema of increasing sequences. Therefore the quotient spaces $\ellp\coloneqq\pintegrablelfun/\mathcal{N}$ and $\ebllq \coloneqq \eblfun / \aezero$ are Archimedean vector lattices. We shall write $[f]$ for the image of  $f\in\pintegrablelfun$ ($1 \leq p \leq \infty$) under the quotient map. As in the classical case, the map defined by $\pnorm{[f]}=\pnorm{f}$ for any $[f]\in\ellp$ is an $\ls$-norm.

The proof that $\ebllq$ is sequentially complete is as in the classical case; modulo almost everywhere equivalence it is \cite[Theorem~3.5.2]{L-functional_analysis}. The situation for $\ellp$ is more complicated. We will first show that every absolutely converging series in $\ellp$ converges.

\begin{thm}\label{Lp_is_sigma_complete-L-normed_space}
	Let $1\leq p< \infty$. Then any absolutely converging series in $\ellp$ converges.
\end{thm}

\begin{proof}
	Let $\set{f_k}_{k=1}^\infty\subseteq\pintegrablelfun$ such that $\displaystyle\sum_{k=1}^\infty\pnorm{f_k}=\bigvee_{N=1}^{\infty}\sum_{k=1}^N\pnorm{f_k}$ exists in $\posls$; we denote its limit by $\lambda$.
	Let $G_n\coloneqq \displaystyle\sum_{k=1}^n\abs{f_k}$ for every $n\in\N$. Then $(G_n)$ is a monotone increasing sequence in $\pintegrablelfun$. We define $G$ as the pointwise limit of $(G_n)$.  Then $G_n^p\uparrow G^p$ pointwise. By the monotone convergence theorem, \Cref{res:MCT_increasing}, $G$ and $G^p$ are measurable and
 $$\int_S G^p \d\mu = \bigvee_{n=1}^\infty \int_S G_n^p \d\mu.$$
Using \partref{part:inequality_in_L-integral_1} of \Cref{res:inequality_in_L-integral}  we obtain
$$ \bigvee_{n=1}^\infty \int_S G_n^p \d\mu = \bigvee_{n=1}^\infty \int_S \left(\sum_{k=1}^n\abs{f_k}\right)^p \d\mu \leq \bigvee_{n=1}^\infty \sum_{k=1}^n \int_S \abs{f_k}^p \d\mu \leq \lambda^p.$$
Hence $G^p\in\integrablelfun$, i.e.  $G\in\pintegrablelfun$.

By \Cref{res:positive_integrable_is_almost_everywhere_finite}, $G(s)\in\ls$ almost everywhere, so for $\mss\coloneqq\set{s\colon G(s)\in\ls}$ we have that $\npm(A^c)=0$. We define, for all $n \geq 1$,
 $$F_n\coloneqq \sum_{k=1}^nf_k\indicator{\mss}+0\indicator{\mss^c} .$$
Then each $F_n$ is in $\lmeasfun$ and $\abs{F_n}\leq G$.
Since for all $m>n$,
		$$\abs{F_m(s)-F_n(s)}\leq
		\begin{cases}
		\sum_{k=n+1}^m \abs{f_k(s)} & \text{ if }s\in\mss\\
		0 &  \text{ if } s\in\mss^c,
		\end{cases} $$
	 $(F_n(s))$ order converges in $\ls$ for all $s\in\pset$. We let $F$  to be the pointwise limit of $(F_n)$, i.e.
 \[F(s)\coloneqq  \begin{cases}
 \sum_{k=1}^{\infty}f_k(s)\indicator{\mss}(s)\ & \text{ if }s\in\mss,\\
 0 & \text{ if }s\in\mss^c.
 \end{cases}\]
  Since $\lmeasfun$ is \sDc\  (see \Cref{res:measurable_function_space_is_sDc}), $F$ is measurable and $\abs{F}\leq G$.

By \cite[Lemma~4.0.2$(ii)$]{L-functional_analysis}, $\abs{F-F_n}^p\rightarrow 0$ pointwise, and $\abs{F-F_n}^p\leq (2G)^p$, so it follows from the Dominated convergence theorem, \Cref{res:DCT}, that $\displaystyle \int_S \abs{F-F_n}^p \d\mu \rightarrow 0$.

 We now verify that $F$ is the $\pnorm{\cdot}$-limit of $\sum_{k=1}^n f_k$.
Note that for every $n\in \N$,
\begin{align*}
 \norm{F - \sum_{k=1}^n f_k}_p^p &= \int_S \abs{F-\sum_{k=1}^nf_k}^p \d\mu\\
 &= \int_S \abs{F\indicator{\mss}-\sum_{k=1}^nf_k\indicator{\mss}}^p \d\mu + \int_S \abs{\mathbf{0}-\sum_{k=1}^nf_k\indicator{\mss^c}}^p \d\mu.
 \end{align*}
By the above we have that
\[
\int_S \abs{F\indicator{\mss}-\sum_{k=1}^nf_k\indicator{\mss}}^p \d\mu = \int_S \abs{F-F_n}^p \d\mu \rightarrow 0, \]
and furthermore,
\[0\leq \int_S \abs{0-\sum_{k=1}^nf_k\indicator{\mss^c}}^p \d\mu \leq \int_S \sum_{k=1}^n\abs{f_k}^p\indicator{\mss^c} \d\mu \leq \sum_{k=1}^n\pnorm{f_k}^p\npm(\mss^c)=0.\]
Therefore $\pnorm{F-\sum_{k=1}^\infty f_k}^p \rightarrow 0$ and so $\pnorm{F-\sum_{k=1}^\infty f_k} \rightarrow 0$ by \cite[Lemma~4.0.2 $(ii)$]{L-functional_analysis}.
	\end{proof}

\section{Completeness}\label{sec:sequential_completeness}

In the classical case, it would follow immediately from \Cref{Lp_is_sigma_complete-L-normed_space} that $\ellp$ is complete, since a (real or complex) normed space is complete if and only if every absolute convergent series is convergent.  However, it is unknown whether this equivalence holds for general $\ls$-normed spaces.  The following results gives sufficient conditions on $\ls$ so that any $\ls$-normed space in which every absolutely converging sequence converges is complete.

We first show that sequential completeness is equivalent to completeness whenever $\L$ has the countable sup property.

\begin{thm}\label{t:  L-normed space is complete iff sequentially complete}
Suppose that $\L$ satisfies the countable sup property.  Let $X$ be an $\L$-normed space.  Then $X$ is complete if and only if $X$ is sequentially complete.
\end{thm}

\begin{proof}
Clearly, if $X$ is complete, then it is sequentially complete.  Assume that $X$ is sequentially complete, and let $(x_\alpha)$ be a Cauchy net in $X$.  Then there exists a set $\Eps \subseteq \L^+$ such that $\inf \Eps = 0$, and, for every $\eps\in\Eps$ there exists $\alpha_\eps$ so that $\norm{x_\alpha-x_\beta}\leq \eps$ for all $\alpha,\beta \geq \alpha_\eps$.  Since $\L$ satisfies the countable sup property, there exists a sequence $(\eps_n)$ in $\Eps$ so that $\inf\{\eps_n ~:~ n\in\N\}=0$.  Select an increasing sequence of indices, $(\alpha_n)$, such that for every $n\in \N$, if $\alpha,\beta \geq \alpha_n$ then $\norm{x_\alpha-x_\beta}\leq \eps_n$.  Then $(x_{\alpha_n})$ is a Cauchy sequence in $X$.  By assumption, there exists $x\in X$ so that $x_{\alpha_n}\to x$.  For all $n\in \N$ and $\alpha \geq \alpha_n$ we have $\norm{x-x_\alpha} \leq \norm{x-x_{\alpha_n}}+\norm{x_{\alpha_n}-x_\alpha} \leq \norm{x-x_{\alpha_n}} + \eps_n$.  From this it follows that $x_\alpha \to x$.
\end{proof}

For the next result we need the following theorem about locally solid spaces, see \cite[Theorem 2.21]{aliprantis_burkinshaw_LOCALLY_SOLID_RIESZ_SPACES_WITH_APPLICATIONS_TO_ECONOMICS_SECOND_EDITION:2003}.

\begin{thm}\label{Thm:  Limit of increasing convergent sequence is supremum}
Let $E$ be a locally solid vector lattice, $(x_\alpha)$ an increasing net in $E$ and $x\in E$.  If $(x_\alpha)$ converges to $x$ then $x_\alpha \uparrow x$.
\end{thm}

\begin{thm}\label{t:series_completeness_implies_completeness}
Assume that $\L$ satisfies the countable sup property and admits a $\sigma$-order continuous completely metrizable locally solid topology.  Let $X$ be an $\L$-normed space such that every absolutely convergent series in $X$ converges. Then $X$ is complete.
\end{thm}

\begin{proof}
Let $d$ be a complete, translation invariant metric that induces a $\sigma$-order continuous topology on $\L$. According to \Cref{t:  L-normed space is complete iff sequentially complete}, it suffices to show that $X$ is sequentially complete.  Let $(x_n)$ be a Cauchy sequence in $X$, see \cite[Definition~3.2.1]{L-functional_analysis}. Since $\L$ has the countable sup property, there exists a sequence $(\eps_k)$ in $\L^+$ such that $\inf_k \eps_k = 0$ and for each $k \in \N$, there exists $n_k \in \N$ with $\norm{x_n  - x_m} \leq \eps_k$ for all $n,m \geq n_k$. By passing to finite infima we may assume that $(\eps_k)$ decreases to $0$. Since $d$ is $\sigma$-order continuous, $d(0, \eps_k) \to 0$, and by passing to a subsequence we may assume that $d(0, \eps_k) < 2^{-k}$ for every $k\in \N$ so that $\sum_{k=1}^\infty \eps_k$ converges by completeness of $d$. By \Cref{Thm:  Limit of increasing convergent sequence is supremum}, $\sum_{k=1}^\infty \eps_k$ converges in order.

We may assume that $(n_k)$ is strictly increasing and so $(x_{n_k})$ is a subsequence of $(x_n)$. For $k \geq 1$, define $y_k \coloneq x_{n_k} - x_{n_{k-1}}$ , where we let $x_0 = 0$.  Since
\[
\sum_{k=2}^\infty \norm{y_k} = \sum_{k=2}^\infty \norm{x_{n_k} - x_{n_{k-1}}} \leq \sum_{k=2}^\infty \eps_{k-1},
\]
which converges in order, $\sum_{k=1}^\infty y_k$ is an absolutely convergent series; hence it converges by assumption. The partial sums of $\sum_k y_k$ equals $x_{n_k}$ and so $(x_{n_k})$ converges. Therefore the Cauchy sequence $(x_n)$ has a converging subsequence, and the proof that such a sequence converges is exactly as in the classical case.
\end{proof}

Combining \Cref{t:series_completeness_implies_completeness}, \Cref{t:L^0_has_nice_topology}, and \Cref{Lp_is_sigma_complete-L-normed_space} yields the following.

\begin{corollary}\label{Cor:Lp sequentially complete}
    Let $1 \leq p < \infty$ and suppose that $\L = L^0(m)$ for some $\sigma$-finite measure $m$. Then $\ellp$ is complete.
\end{corollary}

We remark that \Cref{t:  L-normed space is complete iff sequentially complete} is related to a general result in the theory of (real or complex) convergence vector spaces:  A first countable convergence vector space is complete if and only if it is sequentially complete, see for instance \cite[Proposition 3.6.5]{beattie_butzmann_CONVERGENCE_STRUCTURES_AND_APPLICATIONS_TO_FUNCTIONAL_ANALYSIS:2002}.  In the case of an Archimedean vector lattice $X$ and order convergence, this reduces to the well known fact that if $X$ has the countable sup property, then it is Dedekind complete if and only it is complete with respect to order convergence, if and only if it is Dedekind $\sigma$-complete, see e.g. \cite[Proposition 9.10]{OBrien_vanderwalt_troitsky_convergence_structures} and \cite[Theorem 23.6]{luxemburg_zaanen_RIESZ_SPACES_VOLUME_I:1971}.

\subsection*{Acknowledgements}

The authors thank Marcel de Jeu for helpful discussions.  Part of this work was completed during a visit of Marten Wortel to Sichuan University. This visit was made possible by the financial support of the National Natural Science Foundation of China, Grant number 12201439.
	
\bibliographystyle{plain}
\urlstyle{same}

\bibliography{L-valued_integration}

\end{document}